\documentclass{amsart}
\usepackage[leqno]{amsmath}
\usepackage{amsthm}
\usepackage{amsfonts}

\usepackage[OT4]{fontenc}

\usepackage[margin={3cm,4cm}]{geometry}

\theoremstyle{plain}
\newtheorem{lem}{Lemma}[section]
\newtheorem{thm}{Theorem}

\theoremstyle{definition}
\newtheorem{defi}[lem]{Definition}

\theoremstyle{remark}
\newtheorem{rem}[lem]{Remark}

\newcommand{\ud}{\mathrm{d}}
\newcommand{\Ud}{\mathrm{D}}

\providecommand{\abs}[1]{\left\lvert#1\right\rvert}
\providecommand{\norm}[1]{\left\lVert#1\right\rVert}
\providecommand{\floor}[1]{\lfloor#1\rfloor}

\numberwithin{equation}{section}

\DeclareMathOperator{\dist}{dist}

\DeclareMathOperator{\supp}{supp}

\begin{document}

\title[Boundary-value problem of the heat equation]{Very weak solutions to the boundary-value problem of the homogenous heat equation}
\author{Bernard Nowakowski}
\author{Wojciech M. Zaj\k aczkowski}
\thanks{Research of both authors is partially supported by Polish KBN grant NN 201 396937}
\address{Bernard Nowakowski\\ Institute of Mathematics\\ Polish Academy of Sciences\\ \'Snia\-deckich 8\\ 00-956 Warsaw\\ Poland}
\email{bernard@impan.pl}
\address{Wojciech M. Zaj\c aczkowski\\ Institute of Mathematics\\ Polish Academy of Sciences\\ \'Sniadeckich 8\\ 00-956 Warsaw\\ Poland\\ and \\ Institute of Mathematics and Crypto\-logy\\ Military University of Technology\\ Kaliskiego 2\\ 00-908 Warsaw\\ Poland}
\email{wz@impan.pl}
\subjclass[2000]{35K05, 35K20}
\keywords{very weak solutions, integral equation, non-regular solutions, anisotropic Sobolev-Slobodecki spaces, boundary-value problem, heat equation}
\begin{abstract}
    We consider the homogeneous heat equation in a domain $\Omega$ in $\mathbb{R}^n$ with vanishing initial data and the Dirichlet boundary condition. We are looking for solutions in $W^{r,s}_{p,q}(\Omega\times(0,T))$, where $r < 2$, $s < 1$, $1 \leq p < \infty$, $1 \leq q \leq \infty$. Since we work in the $L_{p,q}$ framework any extension of the boundary data and integration by parts are not possible. Therefore, the solution is represented in integral form and is referred as \emph{very weak} solution. The key estimates are performed in the half-space and are restricted to $L_q(0,T;W^{\alpha}_p(\Omega))$, $0 \leq \alpha < \frac{1}{p}$ and $L_q(0,T;W^{\alpha}_p(\Omega))$, $\alpha \leq 1$. Existence and estimates in the bounded domain $\Omega$ follow from a perturbation and a fixed point arguments.
\end{abstract}

\maketitle

\section{Introduction}\label{s1}

We examine the following initial boundary-value problem
\begin{equation}\begin{aligned}\label{p111}
    &u,_t -\triangle u = 0  & &\textrm{in } \Omega\times (0,T) =: \Omega^T,\\
    &u = \varphi & &\textrm{on } S\times(0,T) =: S^T,\\
    &u\vert_{t=0} = 0 & &\textrm{in } \Omega\times\{0\},
    \end{aligned}
\end{equation}
where $\Omega$ is a bounded subset in $\mathbb{R}^n$ with the boundary $S$ or the entire half-space $\mathbb{R}^n_+$. We are intensely interested in the problem of maximal regularity of solutions in dependence on the boundary data. The solvability and the maximal regularity of \eqref{p111} has been studied by many authors under various requirements on the boundary data. Let us briefly outline certain results which are the closest to the intended contribution of this work. For a full summary of the research into solvability of \eqref{p111} we refer the reader to the Introduction in \cite{zad}. 

The classic case, when $\varphi \in W^{2 - \frac{1}{p},1 - \frac{1}{2p}}_p(S^T)$, $p > 1$, was widely studied in \cite[Ch. 4, $\S$3, $\S$4]{lad} (for a different approach see also \cite[Thm. 4.2]{gris}) and then extended in \cite{sol2}. Recently, an analogous result (see \cite[Thm. 2.1]{denk}) was obtained for vector-valued parabolic initial-boundary value problem of general type.  

For anisotropic boundary data it was shown in \cite[Thm. 3.1]{weid} that the maximal regularity of solutions in $W^{2,1}_{p,q}(\Omega^T)$-space (see Definition \ref{don}) can be achieved only when $\varphi \in L_q(0,T;W_p^{2 - \frac{1}{p}}(S)) \cap F^{1 - \frac{1}{2p}}_{q,p}(0,T;L_p(S))$, where $\frac{3}{2} < p \leq q < \infty$ and $F^{\alpha}_{q,p}(0,T;L_p(S))$ is a Lizorkin-Triebel space. Not much later, this result was improved for any $p$ and $q$ satisfying $1 < p, q < \infty$ in the case of general equations of parabolic type (see  \cite[Thm. 2.3]{denk}). 

The aim of this paper is to prove the existence and uniqueness of such solutions to problem \eqref{p111} that have the maximal regularity of $L_q(0,T;L_p(\Omega))$ or $L_q(0,T;W^1_p(\Omega))$, where $1 \leq p < \infty$, $1 \leq q \leq \infty$. In standard approach we could try to incorporate the classical regularizer technique from \cite{sol1} but it requires more regularity for the boundary data than we want to assume. Therefore, we need another approach which is based on the concept of very weak solutions. Also note that the case $W^{1,0}_{2,2}(\Omega^T) = L_2(0,T;H^1(\Omega))$ corresponds to the regularity of weak solutions but the energy estimate in this space cannot be obtained in the standard way. This only confirms that we need another definition of solution to problem \eqref{p111}.

\begin{defi}\label{def1}
    We say that a function $u$ is a very weak solution to the problem \eqref{p111} if and only if it satisfies the following integral equation
    \begin{equation}\label{r30}
        u(x,t) = \int_0^t\!\!\!\int_S n_i(\xi)\cdot \frac{\partial\Gamma(x-\xi,t-\tau)}{\partial \xi_i}\mu(\xi,\tau)\, \ud S_\xi \ud \tau,
    \end{equation}
    where $\mu$ is an unknown function called the density of double layer, which depends on the boundary condition $\varphi$ and it has to be calculated separately, $n$ is the unit outward vector and $\Gamma$ is the fundamental solution to the heat equation and is given by the formula
    \begin{equation*}
        \Gamma(x,t) = \begin{cases}
                          \frac{1}{\big(4\pi t\big)^{\frac{n}{2}}} e^{-\frac{\abs{x}^2}{4t}} & t > 0,\\
                          0 & t < 0.
                      \end{cases}
    \end{equation*}
	As mentioned, the function $\mu$ is a priori unknown but it is a solution to the Fredholm integral equation of second order
	\begin{equation}\label{eq1.3}
		\varphi(\eta,t) = \int_0^1\!\!\!\int_S n_i(\xi) \frac{\partial \Gamma(\eta - \zeta,t - \tau)}{\partial \xi} \mu(\xi,\tau)\ud S_\xi\, \ud \tau - \frac{1}{2}\mu(\eta,t), \qquad \eta \in S,
	\end{equation}
	which we obtain from \eqref{r30} after passing with $x \to \xi \in S$ and using \eqref{p111}$_2$. By $\ud S_\xi$, $\xi \in S$, we denote the measure of $S$.
\end{defi}
For a deeper discussion of the above definition of the solution we refer the reader to \cite[Ch. $4$, $\S$1]{lad}.

To prove the existence of the very weak solutions we solve equation \eqref{eq1.3}. Subsequently, to find the estimates of solutions first we consider the model problem
\begin{equation}\begin{aligned}\label{p1}
    &u,_t -\triangle u = 0  & &\textrm{in } \mathbb{R}^n_{+}\times (0,T),\\
    &u = \varphi & &\textrm{in } \mathbb{R}^{n - 1}\times(0,T),\\
    &u\vert_{t=0} = 0 & &\textrm{in } \mathbb{R}^n_+\times\{t = 0\},
    \end{aligned}
\end{equation}
to which the solution has the form
\begin{equation}\label{eq1.4}
        \begin{aligned}
            u(x,t) &= -2 \int_0^t\!\!\! \int_{\mathbb{R}^{n - 1}} \frac{\partial\Gamma(x' - y',x_n,t - \tau)}{\partial x_n}\varphi(y',\tau)\, \ud y'\, \ud \tau \\
            &= \frac{1}{(4\pi)^{\frac{n}{2}}} \int_0^t\!\!\! \int_{\mathbb{R}^{n - 1}} \frac{x_n}{(t - \tau)^{\frac{n + 2}{2}}}e^{-\frac{\abs{x' - y'}^2 + x_n^2}{4(t - \tau)}}\varphi(y',\tau)\, \ud y'\, \ud \tau,
        \end{aligned}
\end{equation}
where $x' = (x_1,\ldots,x_{n - 1})$ and $y' = (y_1,\ldots,y_{n - 1})$ (for further details see Lemma \ref{l5}) and derive necessary estimates. Next, we introduce a partition of unity with respect to $\Omega$ in \eqref{p111} and use the estimates obtained for the half-space.

Now we can formulate three results of this paper:
\begin{thm}\label{t1}
    Let $1 \leq p < \infty$, $1 \leq q \leq \infty$ and $r,s \geq 0$ be fixed. Suppose that $\varphi \in W^{r,s}_{p,q}(S^T)$. Then there exists a unique function $\mu \in W^{r,s}_{p,q}(S^T)$ (see Definition \ref{def1}) such that $\norm{\mu}_{W^{r,s}_{p,q}(S^T)} \leq c \norm{\varphi}_{W^{r,s}_{p,q}(S^T)}$.
\end{thm}

\begin{thm}\label{t2}
    Suppose that 
    \begin{enumerate}
        \item\label{i1} $p \in \lbrack 1,\infty)$, $q \in \lbrack 1,\infty\rbrack$ and $\varphi \in L_q(0,T;L_p(\mathbb{R}^{n - 1}))$. Then $ u \in L_q(0,T;W^\alpha_p(\mathbb{R}^n_+))$, where $0 \leq \alpha < \frac{1}{p}$ and
            \begin{equation*}
                \norm{u}_{L_q(0,T;W^\alpha_p(\mathbb{R}^n_+))} \leq c(p,T) \norm{\varphi}_{L_q(0,T;L_p(\mathbb{R}^{n - 1}))},
            \end{equation*}
        \item\label{i2} $p \in \lbrack 1,\infty)$, $q \in \lbrack 1,\infty\rbrack$ and $\varphi \in L_q(0,T;W^{1 - \frac{1}{p}}_p(\mathbb{R}^{n - 1}))$. Then $\Ud_{x'}u \in L_q(0,T;L_p(\mathbb{R}^n_+))$ and
            \begin{equation*}
                \norm{\Ud_{x'}u}_{L_q(0,T;L_p(\mathbb{R}^n_+))} \leq c(n,p) \norm{\varphi}_{L_q(0,T;W^{1 - \frac{1}{p}}_p(\mathbb{R}^{n - 1}))},
            \end{equation*}
				where by $\Ud_{x'}$ we mean any first order partial derivative alongside tangent direction.
        \item\label{i3} $p \in \lbrack 1,\infty)$, $q \in \lbrack 1,\infty\rbrack$ and $\varphi \in W_{p,q}^{1 - \frac{1}{p}, \frac{1}{2} - \frac{1}{2p}}(\mathbb{R}^{n - 1}\times (0,T))$. Then $\partial_{x_n}u \in L_q(0,T;L_p(\mathbb{R}^n_+))$ and
            \begin{equation*}
                \norm{\partial_{x_n}u}_{L_q(0,T;L_p(\mathbb{R}^n_+))} \leq c(n,p)\norm{\varphi}_{W_{p,q}^{1 - \frac{1}{p}, \frac{1}{2} - \frac{1}{2p}}(\mathbb{R}^{n - 1}\times (0,T))},
            \end{equation*}
    \end{enumerate}
\end{thm}

\begin{thm}\label{t3}
	Suppose that $p \in \lbrack 1,\infty)$, $q \in \lbrack 1,\infty\rbrack$ and $\varphi \in W_{p,q}^{1 - \frac{1}{p}, \frac{1}{2} - \frac{1}{2p}}(S^T)$. Then $u \in L_q(0,T;W^1_p(\Omega))$ and
	\begin{equation*}
		\norm{u}_{L_q(0,T;W^1_p(\Omega))} \leq c_{n,p,q,\Omega,T}\norm{\varphi}_{L_q(0,T;L_p(S))} + c_{n,p,q,\Omega}\norm{\varphi}_{W_{p,q}^{1 - \frac{1}{p}, \frac{1}{2} - \frac{1}{2p}}(S^T)}.
	\end{equation*}
\end{thm}

The reader has surely noticed that Theorem \ref{t1} concerns the existence of the density of the double layer whereas Theorems \ref{t2} and \ref{t3} provide suitable estimates in the half-space and in bounded domains of solutions in the form \eqref{r30} for $\varphi \in W^{1 - \frac{1}{p}, \frac{1}{2} - \frac{1}{2p}}_{p,q}(S^T)$. The difference between the existence of solutions and their estimates by the boundary data is particularly visible when we compare the function spaces used in all three theorems. Note that Theorem \ref{t1} covers a whole range of anisotropic Sobolev-Slobodecki spaces $W^{r,s}_{p,q}(\Omega^T)$, whereas the claims of Theorems \ref{t2} and \ref{t3} are only restricted to $L_q(0,T;W^1_p(\Omega))$ and $L_q(0,T;L_p(\Omega))$. The reason behind our choice follows from technical difficulties which appear when $r$ and $s$ are non-integers. This case will be covered in forthcoming paper.

The reader can also easily recognize that Theorem \ref{t3} contains less results than Theorem \ref{t2}. The motivation is not to extend the paper and only to show ideas of the proof in the case of a bounded domain.

Theorem \ref{t3} plays a crucial role in proofs concerning the existence of global and regular solutions to the Navier-Stokes equations in cylindrical domains in $\mathbb{R}^3$ (see e.g. \cite[Lemma 4.1]{wm2}, \cite[Appendix]{wm8}, \cite[Lemma 3.1]{ren1}, \cite[Lemma 4.3]{wm6}), because the solvability and the estimates of solutions to the heat equation with the boundary data either from $L_\infty(0,T;L_2(S))$ or from $H^{\frac{1}{2},\frac{1}{4}}(S^T)$ contribute significantly to the improvement in the regularity of weak solutions. 

This paper is divided into five sections. In Sections \ref{s1} and \ref{s2} the reader can find the description of the problem and auxiliary results required to prove all three theorems. Section \ref{s3} is devoted to the existence of solutions to problem \eqref{p111}. It contains the proof Theorem \ref{t1}. In Section \ref{s4} we present various estimates for solutions to problem \eqref{p1} which are stated in Theorem \ref{t2} and in Section \ref{s5} we give the proof of Theorem \ref{t3}.

\section{Auxiliary results}\label{s2}

In this section we collect helpful tools for further calculations and introduce the function spaces that will be used frequently in this paper.

\begin{lem}\label{l4}
    Let $\Gamma(x,t)$ be the fundamental solution of the heat equation. Then
    \begin{equation*}
        \int_{\mathbb{R}^n} \partial^r_t\Ud^s_x \Gamma(x,t)\, \ud x =
            \begin{cases}
                1 & r = s = 0, \\
                0 & r + s \geq 1,
            \end{cases}
    \end{equation*}
	where the symbol $\Ud^s_x$ denotes any derivative of order $s$ with respect to $x$.
\end{lem}
Note that the integral above does not depend on time. The integration is carried out only with respect to spatial variables.

\begin{lem}\label{l5}
    Any solution to the problem \eqref{p1} in the half-space $x_n > 0$ has the form
    \begin{equation}
        \begin{aligned}\label{r2}
            u(x,t) &= -2 \int_0^t\!\!\! \int_{\mathbb{R}^{n - 1}} \frac{\partial\Gamma(x' - y',x_n,t - \tau)}{\partial x_n}\varphi(y',\tau)\, \ud y'\, \ud \tau \\
            &= \frac{1}{(4\pi)^{\frac{n}{2}}} \int_0^t\!\!\! \int_{\mathbb{R}^{n - 1}} \frac{x_n}{(t - \tau)^{\frac{n + 2}{2}}}e^{-\frac{\abs{x' - y'}^2 + x_n^2}{4(t - \tau)}}\varphi(y',\tau)\, \ud y'\, \ud \tau,
        \end{aligned}
    \end{equation}
    where $x' = (x_1,\ldots,x_{n - 1})$ and $y' = (y_1,\ldots,y_{n - 1})$.
\end{lem}

The proofs of Lemmas \ref{l4} and \ref{l5} can be found in \cite[Ch. $4$, $\S 1$]{lad}.

\begin{lem}[the general Minkowski inequality]\label{l3}
    Let $1 \leq p \leq \infty$ and let $f$ be a~measurable function on $\mathbb{R}^n\times\mathbb{R}^m$. Then
    \begin{equation*}
        \left(\int_{\mathbb{R}^n} \abs{\int_{\mathbb{R}^m} f(x,y)\, \ud y}^p\, \ud x\right)^{\frac{1}{p}} \leq \int_{\mathbb{R}^m} \left(\int_{\mathbb{R}^n} \abs{f(x,y)}^p\, \ud x\right)^{\frac{1}{p}}\, \ud y.
    \end{equation*}
\end{lem}

For the detailed proof of Lemma \ref{l3} 
we refer the reader to \cite[Ch. $1$, $\S 2$]{bes}.

\begin{lem}\label{l2}
    We have
    \begin{equation*}
        \sup_{s \geq 0} s^re^{-s} = r^re^{-r}.
    \end{equation*}
\end{lem}

\begin{defi}\label{don}
    We say that a function $f$ belongs to the space $W_{p,q}^{r,s}(\Omega^T)$, where $\Omega\subseteq \mathbb{R}^n$, $p,q > 1$, $r,s \geq 0$, if and only if
    \begin{multline*}
        \norm{f}_{W_{p,q}^{r,s}(\Omega^T)} = \sum_{0\leq r'\leq \floor{r}} \!\!\!\left(\int_0^T\!\!\!\left(\int_\Omega \abs{\Ud^{r'}_x\! f(x,t)}^p\, \ud x\right)^{q/p} \ud t\right)^{1/q} \\
        + \left(\int_0^T\!\!\!\left(\int_\Omega\int_\Omega \frac{\abs{\Ud^{[r]}_x\! f(x,t)-\Ud^{[r]}_y\! f(y,t)}^p}{\abs{x-y}^{n+p(r-[r])}}\, \ud x\, \ud y\right)^{q/p}\ud t\right)^{1/q} \\
        + \sum_{0\leq s'\leq [s]} \!\!\!\left(\int_0^T\!\!\!\left(\int_\Omega \abs{\partial^{s'}_t\! f(x,t)}^p\, \ud x\right)^{q/p} \ud t\right)^{1/q} \\
        + \left(\int_0^T\!\!\!\int_0^T\!\!\!\frac{\left(\int_\Omega \abs{\partial^{[s]}_t\! f(x,t) - \partial^{[s]}_w\! f(x,w)}^p\, \ud x\right)^{q/p}}{\abs{t-w}^{1+q(s-[s])}}\, \ud t\, \ud w\right)^{1/q} < \infty.
    \end{multline*}

    We say that a function $\varphi$ belongs to the space $W_{p,q}^{r,s}(S^T)$, where $S := \partial \Omega$ is a compact manifold (provided $\Omega$ is bounded and open), if all functions
    \begin{equation*}
        (\varphi\beta_i)\circ \alpha_i^{-1}\colon\alpha_i(U_i)\to\mathbb{R}, \qquad i\in I
    \end{equation*}
     belong to $\overset{\circ}{W}\,\!^{r,s}_{p,q}(\alpha_i(U_i))$, which is understood as the closed hull of $\mathcal{D}(S)$ in $W^{r,s}_{p,q}(S^T)$, whereas $(U_i,\alpha_i)$ is an admissible $\mathcal{C}^r$-atlas for $S$ and $\beta_i$ is a subordinate partition of unity. In this case $S\cap U_i$ is given by the equation $x_n = f_i(x_1,\ldots,x_{n - 1})$ and the derivatives $D^{r'}$ and $D^{\floor{r}}$ are taken with respect to the variables $x_1,\ldots,x_{n - 1}$.

     For more details, see \cite[Ch. $1$, $\S 3$ and $\S 4$]{wlo}.
\end{defi}

\begin{lem}\label{l9}
    Suppose that $\psi \in \mathcal{C}^{r}(S)$ and $\mu \in W^{r,s}_{p,q}(S^T)$. Then $\psi\mu \in W^{r,s}_{p,q}(S^T)$ and there exists a constant $c$ such that
    \begin{equation*}
        \norm{\psi \mu}_{W^{r,s}_{p,q}(S^T)} \leq c \norm{\mu}_{W^{r,s}_{p,q}(S^T)}
    \end{equation*}
    which depends on $r$ and $\psi$.
\end{lem}

For proof of this Lemma, see \cite[Ch. $1$, $\S 4$, Proof of Proposition $4.5$; there are slight differences, but they are related to technical details]{wlo}.

\section{The existence of solutions --- Proof of Theorem \ref{t1}}\label{s3}

In this section we prove the existence of solutions to problem \eqref{p111}. The proof is based on the method of successive approximations. The beginning point is formula \eqref{eq1.3}, which we rewrite in the form

\begin{equation}\label{r21}
	\mu(\xi,t) - \int_0^t\!\!\!\int_S N(\xi,t;\eta,\tau)\mu(\eta,\tau)\, \ud S_\eta\, \ud \tau = g(\xi,t),
\end{equation}
where $g(\xi,t) = -2\varphi(\xi,t)$ and $N(\xi,t;\eta,\tau)$ is given by 
\begin{equation*}
	N(\xi,t;\eta,\tau) = 2\frac{\partial}{\partial n_{\eta}} \Gamma(\xi - \eta, t - \tau) = \frac{\abs{\xi - \eta}\cos(\xi - \eta,n_{\eta})}{t - \tau}\Gamma(\xi - \eta, t - \tau).
\end{equation*}
Therefore \eqref{r21} can be regarded as a product of a weakly singular kernel of Volterra type and of a weakly singular kernel. By $n_{\eta}$ we understand the unit outward normal vector in point $\eta \in S$ and $\frac{\partial}{\partial n_{\eta}}$ denotes the normal derivative with respect to variable $\eta$.

Next we solve equation \eqref{r21}. Since the kernel $N(\xi,t;\eta,\tau)$ is unbounded we first iterate this equation so many times that the obtained iterated equation possesses a bounded kernel. However, two natural questions may arise: 1. Do the solutions of the iterated and original equations coincide? 2. How do we know that the iterated equation have a bounded kernel? The answers to this questions are guaranteed by two lemmas that we state below:

\begin{lem}
	There exists an integer number $m_0$ such that for any integer $m > m_0$ every solution to the $m$-times iterated integral equation of the given weakly singular equation is the solution of the original equation.
\end{lem}

\begin{lem}
	If the singular kernel $K(x,t;y,\tau)$ has the form
	\begin{equation*}
		K(x,t;y,\tau) = \frac{k(x,t;y,\tau)}{(t - \tau)^\alpha},
	\end{equation*}
	where $k(x,t;y,\tau)$ is bounded and continuous function, then there always exists an integer number $m_0$ dependent on $\alpha$ such that for $m > m_0$ the iterated kernels $K_m(x,t;y,\tau)$ are bounded.
\end{lem}
Both Lemmas are proved \cite[Ch. 3, \S 3]{pog} and \cite[Ch. 3, \S 2]{pog} respectively.

After the first iteration of equation \eqref{r21} we obtain
\begin{align*}
	&\mu(\xi,t) - \int_0^t\!\!\!\int_S N_2(\xi,t;\eta,\tau)\mu(\eta,\tau)\, \ud S_\eta\, \ud \tau = g_2(\xi,t), \\
	&g_2(\xi,t) = g(\xi,t) - \int_0^t\!\!\!\int_S N_1(\xi,t;\eta,\tau)g(\eta,\tau)\, \ud S_\eta\ud \tau
\end{align*}
and $N_1(\xi,t;\eta,\tau) = N(\xi,t;\eta,\tau)$,
\begin{equation*}
	N_2(\xi,t;\eta,\tau) = \int_0^t\!\!\!\int_S N_1(\xi,t;\alpha,s)N_{1}(\alpha,s;\eta,\tau)\, \ud S_\alpha\, \ud s.
\end{equation*}
After $l$ iterations we get
\begin{equation}\label{r23}
	\begin{aligned}
		&\mu(\xi,t) - (-1)^l\int_0^t\!\!\!\int_S N_l(\xi,t;\eta,\tau)\mu(\eta,\tau)\, \ud S_\eta\, \ud \tau = g_l(\xi,t), \\
		&g_l(\xi,t) = g(\xi,t) - \int_0^t\!\!\!\int_S N_1(\xi,t;\eta,\tau)g_{l - 1}(\eta,\tau)\, \ud S_\eta\ud \tau,
	\end{aligned}
\end{equation}
where
\begin{equation*}
	N_l(\xi,t;\eta,\tau) = \int_0^t\!\!\!\int_S N_1(\xi,t;\alpha,s)N_{l - 1}(\alpha,s;\eta,\tau)\, \ud S_\alpha\, \ud s
\end{equation*}
and
\begin{equation*}
	N_l(\xi,t;\eta,\tau) = \frac{e_l(\xi,t;\eta,\tau)}{(t - \tau)^{\frac{n - l + 1}{2}}}e^{-\frac{\abs{\xi - \eta}^2}{4(t - \tau)}},
\end{equation*}
where $e_l(\xi,t;\eta,\tau)$ is a bounded and continuous function for $t\geq \tau$.

Now we see that if $l \geq n + 1$ then the iterated equation has a bounded kernel. Therefore we can apply the method of successive approximations. We finally have
\begin{lem}\label{l10}
	Let us rewrite equation \eqref{r23} in the form
	\begin{equation}\label{row_calk}
		\mu(\xi,t) = g_l(\xi,t) + N_l\mu(\xi,t),
	\end{equation}
	where
	\begin{equation*}
		N_l\mu(\xi,t) = (-1)^l\int_0^t\!\!\!\int_S N_l(\xi,t;\eta,\tau)\mu(\eta,\tau)\, \ud S_\eta \ud \tau.
	\end{equation*}
	Assume that $g_l\in W_{p,q}^{r,s}(S^T)$. Then there exists a unique solution $\mu$ to the above equation such that $\mu \in W_{p,q}^{r,s}(S^T)$.
\end{lem}

\begin{proof}
    Let $\mu_0 = 0$ be the first approximation. Then we get the following sequence for $\mu_n$:
    \begin{equation}\label{soa}
        \begin{aligned}
            \mu_1(\xi,t) &= g_l(\xi,t) + N_l\mu_0(\xi,t) = g_l(\xi,t), \\
            \mu_2(\xi,t) &= g_l(\xi,t) + N_l\mu_1(\xi,t) = g_l(\xi,t) + N_lg_l(\xi,t), \\
            &\vdots \\
            \mu_n(\xi,t) &= g_l(\xi,t) + N_l\mu_{n-1}(\xi,t) =\\
            &\mspace{40mu} g_l(\xi,t) + N_lg_l(\xi,t) + N^2_lg_l(\xi,t) + \ldots + N^{n-1}_lg_l(\xi,t), \\
            &\vdots
        \end{aligned}
    \end{equation}
    where the $N^k_l$ iteration of $N_l$ is given by the formula
    \begin{equation*}
        N^k_lg_l(\xi,t) = N_l(N^{k-1}_lg_l)(\xi,t) = (-1)^l\int_0^t\!\!\!\int_S N_l(\xi,t;\eta,\tau)N^{k-1}_lg_l(\eta,\tau)\, \ud S_\eta\, \ud \tau
    \end{equation*}
    for $k = 2,3,\ldots$ and
    \begin{equation*}
        N^1_lg_l(\xi,t) = N_lg_l(\xi,t) = (-1)^l\int_0^t\!\!\!\int_S N_l(\xi,t;\eta,\tau)g_l(\eta,\tau)\, \ud S_\eta\, \ud \tau.
    \end{equation*}

    We see that \eqref{soa} builds a sequence of partial sums of the Neumann series
    \begin{equation*}
        g_l(\xi,t) + N_lg_l(\xi,t) + N^2_lg_l(\xi,t) + N^3_lg_l(\xi,t) + \ldots + N^n_lg_l(\xi,t) + \ldots
    \end{equation*}
    We will show that the above series converges in the norm of the space $W_{p,q}^{r,s}(S^T)$ for $\xi \in S$ and $t\in(0,T)$, $T>0$, by checking the Cauchy condition. Since the norm of the space $W_{p,q}^{r,s}(S^T)$ consists of four different terms (see Definition \ref{don}), each term needs to be treated separately.

    Let $m > n$ and let us introduce the quantity $\abs{N_l}$ by the formula
    \begin{equation}\label{N}
        \abs{N_l} := \sup_{\substack{\xi,\eta \in S,\\ 0\leq t,\tau \leq T}} \abs{N_l(\xi,t;\eta,\tau)}.
    \end{equation}
    Considering the difference
    \begin{equation}\label{diff}
        \norm{\mu_m-\mu_n}_{W_{p,q}^{r,s}(S)}
    \end{equation}
    and using Lemma \ref{reg_osz} we can estimate the first term ($m=r'$, $n=0$) by
    \begin{equation*}
        \sum_{0\leq r'\leq \floor{r}} \!\!\!\left(\int_0^T\!\!\!\left(\int_S \abs{\sum_{k=n+1}^m\Ud^{r'}_\xi\! N^k_lg(\xi,t)}^p\ud S_\xi\right)^{q/p} \ud t\right)^{1/q} \leq \sum_{0\leq r'\leq \floor{r}} \sum_{k = n + 1}^m \abs{\Phi_{r'}} \abs{N_l}^k \abs{S}^k \frac{T^k}{k!}
    \end{equation*}
    and the third term ($m=0$, $n=s'$) by
    \begin{equation*}
        \sum_{0\leq s'\leq \floor{s}} \!\!\!\left(\int_0^T\!\!\!\left(\int_S \abs{\sum_{k = n + 1}^m\partial^{s'}_t\! N^k_lg(\xi,t)}^p\ud S_\xi\right)^{q/p} \ud t\right)^{1/q} \leq \sum_{0\leq s'\leq \floor{s}} \sum_{k = n + 1}^m \abs{\Phi_{s'}} \abs{N_l}^k \abs{S}^k \frac{T^k}{k!}.
    \end{equation*}
    Next we estime the second and the fourth term by applying Lemma \ref{niereg_osz} and Remark \ref{czw_wyr}. Finally the difference \eqref{diff} is estimated by
    \begin{multline*}
        \sum_{k=n+1}^m\Biggl(\sum_{0\leq r'\leq \floor{r}} \abs{\Phi_{r'}} \abs{N_l}^k \abs{S}^k \frac{T^k}{k!} + \sum_{0\leq s'\leq \floor{s}} \abs{\Phi_{s'}} \abs{N_l}^k \abs{S}^k \frac{T^k}{k!}\\
        + \abs{\Psi} \abs{N_l}^k \abs{S}^k \frac{T^k}{k!} + 2 \abs{\Theta} \abs{N_l}^k \abs{S}^k \frac{T^k}{k!} \Biggr)\\
        = \sum_{k=n+1}^m\Biggl(\sum_{0\leq r'\leq \floor{r}} \abs{\Phi_{r'}} + \sum_{0\leq s'\leq \floor{s}} \abs{\Phi_{s'}} + \abs{\Psi} + \abs{\Theta}\Biggr)\abs{N_l}^k \abs{S}^k \frac{T^k}{k!},
    \end{multline*}
    which can be an arbitrary small number if only $n,m$ are large enough. Therefore the Cauchy condition for the Neumann series is satisfied. The space $W^{r,s}_{p,q}(S^T)$ is complete, hence there exists a limit $\mu$. To end the proof we must only check that $\mu$ solves the equation \eqref{row_calk} and it is unique.

    Indeed, $\mu$ solves \eqref{row_calk}. Consider the equality which defines the successive approximations:
    \begin{equation*}
        \mu_n = g_l + N_l\mu_{n-1}
    \end{equation*}
    and let
    \begin{equation*}
        h = g_l + N_l\mu.
    \end{equation*}
    Substracting the second equation from the first and applying the norm of the space $W^{r,s}_{p,q}(S^T)$ yields
    \begin{equation*}
        \norm{\mu_n - h}_{W^{r,s}_{p,q}(S^T)} \leq \norm{N_l}_{L_\infty(S^T)} \norm{\mu_{n-1}-\mu}_{W^{r,s}_{p,q}(S^T)}.
    \end{equation*}
    Since $\lim_n \mu_n = \mu$ a.e. in the space $W^{r,s}_{p,q}(S^T)$, so $h = \mu$ a.e. and in fact $\mu$ solves the equation \eqref{row_calk}.

    To prove the uniqueness, suppose that $\bar \mu$ is another solution of the equation \eqref{row_calk}. Let $\psi = \mu - \bar\mu$. Then $\psi$ satisfies the following homogenous equation
    \begin{equation*}
        \psi = N_l\psi.
    \end{equation*}
    After $n - 1$ iterations we get
    \begin{equation*}
        \psi = N_l^n\psi.
    \end{equation*}
    Taking the norm of the space $W^{r,s}_{p,q}(S^T)$ and using the estimates from Lemmas \ref{reg_osz}, \ref{niereg_osz} and Remark \ref{czw_wyr} we get
    \begin{equation*}
        \norm{\psi}_{W^{r,s}_{p,q}(S^T)} \leq \norm{\psi}_{W^{r,s}_{p,q}(S^T)} \biggl(\sum_{0\leq r'\leq \floor{r}} \abs{\Phi_{r'}} + \sum_{0\leq s'\leq \floor{s}} \abs{\Phi_{s'}} + \abs{\Psi} + \abs{\Theta}\Biggr)\abs{N_l}^k \abs{S}^k \frac{T^k}{k!}.
    \end{equation*}
    Since $\lim_n \frac{a^n}{n!} = 0$ for $a\in\mathbb{R}$ the right-hand side tends to zero as $k\to\infty$. Hence $\norm{\psi}_{W^{r,s}_{p,q}(S^T)}$ must be zero and this implies that the solution $\mu$ to the equation \eqref{row_calk} is unique. This concludes the proof.
\end{proof}

Below we demonstrate various estimates we used in the above proof.

\begin{lem}\label{reg_osz}
    Let $\abs{N_l}$ be defined as in \eqref{N} and let
    \begin{equation*}
        \abs{\Phi_{m,n}} := \norm{\Ud^m_\eta\partial^n_\tau g_l}_{L_q(0,T;L_p(S))},
    \end{equation*}
    where $m,n \in \mathbb{N} = \{0,1,2,\ldots\}$. Then for any $k=1,2,\ldots$, $0 < \bar t \leq T$
    \begin{equation}\label{teza1}
         \left(\int_0^{\bar t}\!\!\!\left(\int_S \abs{\Ud^m_\xi \partial^n_t N^k_lg_l(\xi,t)}^p \ud S_\xi\right)^{q/p} \ud t\right)^{1/q} \leq \abs{\Phi_{m,n}} \abs{N_l}^k \abs{S}^k \frac{\bar t^k}{k!}.
    \end{equation}
\end{lem}

\begin{proof}
    Let $k=1$. Then
    \begin{multline*}
        \left(\int_0^{\bar t}\!\!\!\left(\int_S \abs{\Ud^m_\xi \partial^n_t N_lg_l(\xi,t)}^p \ud S_\xi\right)^{q/p} \ud t\right)^{1/q} \\
        =\Bigg(\int_0^{\bar t}\!\!\!\bigg(\int_S \Big\lvert\int_0^t\!\!\!\int_S \Ud^m_\xi \partial^n_t N_l(\xi-\eta,t-\tau)g_l(\eta,\tau)\, \ud S_\eta\, \ud \tau\Big\rvert^p \ud S_\xi\bigg)^{q/p} \ud t\Bigg)^{1/q}\\
        = \Bigg(\int_0^{\bar t}\!\!\!\bigg(\int_S \Big\lvert\sum_{\alpha}\int_0^t\!\!\!\int_{S^{\alpha}} N_l(\eta,\tau) \Ud^m_\xi \partial^n_t \big(\psi^{\alpha}(\xi - \eta)g_l(\xi - \eta,t - \tau)\big)\, \ud S_\eta\, \ud \tau\Big\rvert^p \ud S_\xi\bigg)^{q/p} \ud t\Bigg)^{1/q}\\
        \leq \norm{\Ud^m_\xi \partial^n_t \! g_l}_{L_q(0,T;L_p(S))} \sup_{\eta\in S,0\leq \tau \leq T} \abs{N_l(\eta,\tau)} \abs{S} \bar t \leq \abs{\Phi_{m,n}}\abs{N_l}\abs{S}\frac{\bar t^1}{1!}.
    \end{multline*}
    In the last equality we applied the partition of unity, $\sum_\alpha \psi^{\alpha}$ on $S$ in order to integrate by parts. In the last inequality we used the general Minkowski inequality (Lemma \ref{l3}) and Lemma \ref{l9}.

	Suppose now that for a given $k$ \eqref{teza1} holds. We will show that it is valid for $k+1$. We have
	\begin{multline*}
		\left(\int_0^{\bar t}\!\!\!\left(\int_S \abs{\Ud^m_\xi \partial^n_t N^{k+1}_lg(\xi,t)}^p \ud S_\xi\right)^{q/p} \ud t\right)^{1/q} \\
		= \Bigg(\int_0^{\bar t}\!\!\!\bigg(\int_S \Big\lvert\int_0^t\!\!\!\int_S \Ud^m_\xi \partial^n_t N_l(\xi-\eta,t-\tau)N^k_lg_l(\eta,\tau)\, \ud S_\eta\, \ud \tau\Big\rvert^p \ud S_\xi\bigg)^{q/p} \ud t\Bigg)^{1/q}
	\end{multline*}
	which, after introducing a partition of unity $\sum_\alpha \psi^{\alpha}$ on $S$, is equal to
	\begin{multline*}
		\Bigg(\int_0^{\bar t}\!\!\!\bigg(\int_S \Big\lvert\sum_{\alpha}\int_0^t\!\!\!\int_S N_l(\eta,\tau) \Ud^m_\xi \partial^n_t \big(\psi^\alpha(\eta)N^k_lg_l(\xi - \eta,t - \tau)\big)\, \ud S_\eta\, \ud \tau\Big\rvert^p \ud S_\xi\bigg)^{q/p} \ud t\Bigg)^{1/q}\\
		\leq c\int_0^{\bar t}\!\!\!\int_S N_l(\eta,\tau) \left(\int_0^{\bar t}\!\!\!\left(\int_S \abs{\Ud^m_\xi \partial^n_t N^k_lg_l(\xi - \eta,t - \tau)}^p \ud S_\xi\right)^{q/p} \ud t\right)^{1/q}\, \ud S_\eta\, \ud \tau\\
		\leq \abs{N_l} \int_0^{\bar t}\!\!\!\int_S \abs{\Phi_{m,n}} \abs{N_l}^k \abs{S}^k \frac{(\bar t - \tau)^k}{k!} \, \ud S_\eta\, \ud \tau = \abs{\Phi_{m,n}} \abs{N_l}^{k+1} \abs{S}^{k+1} \frac{\bar t^{k+1}}{(k+1)!}.
	\end{multline*}
	The first inequality above is due to the general Minkowski inequality and Lemma \ref{l9}. The last equality ends the proof.
\end{proof}

\begin{lem}\label{niereg_osz}
    Let $\abs{N_l}$ be defined as in \eqref{N} and let
    \begin{equation*}
        \abs{\Psi} := \left(\int_0^T\!\!\!\left(\int_S\int_S \frac{\abs{\Ud^{\floor{r}}_\xi\! g_l(\xi,t)-\Ud^{\floor{r}}_{\xi'}\! g_l(\xi',t)}^p}{\abs{\xi-\xi'}^{n+p(r-\floor{r})}}\, \ud S_\xi\, \ud S_{\xi'}\right)^{q/p}\ud t\right)^{1/q} + \abs{\Phi_{\floor{r},0}},
    \end{equation*}
    where $r\in\mathbb{R}^+$. Then for any $k=1,2,\ldots$ we have
    \begin{equation}\label{teza2}
        \left(\int_0^T\!\!\!\left(\int_S\int_S \frac{\abs{\Ud^{\floor{r}}_\xi\! N^k_lg_l(\xi,t)-\Ud^{\floor{r}}_{\xi'}\! N^k_lg_l(\xi',t)}^p}{\abs{\xi-\xi'}^{n+p(r-\floor{r})}}\, \ud S_\xi\, \ud S_{\xi'}\right)^{q/p}\ud t\right)^{1/q} \leq \abs{\Psi} \abs{N_l}^k \abs{S}^k \frac{t^k}{k!}.
    \end{equation}
\end{lem}

\begin{proof}
	Let $k=1$. Then
    \begin{multline*}
        \left(\int_0^T\!\!\!\left(\int_S\int_S \frac{\abs{\Ud^{\floor{r}}_\xi\! N_lg_l(\xi,t)-\Ud^{\floor{r}}_{\xi'}\! N_lg_l(\xi',t)}^p}{\abs{\xi-\xi'}^{n+p(r-\floor{r})}}\, \ud S_\xi\, \ud S_{\xi'}\right)^{q/p}\ud t\right)^{1/q}\\
    = \Biggl(\int_0^T\!\!\!\biggl(\int_S\int_S \biggl\lvert\int_0^t\!\!\!\int_S \Ud^{\floor{r}}_\xi\! N_l(\xi-\eta,t-\tau)g_l(\eta,\tau)\, \ud S_\eta\, \ud \tau\\
    - \int_0^t\!\!\!\int_S\Ud^{\floor{r}}_{\xi'}\! N_l(\xi'-\eta,t-\tau)g_l(\eta,\tau)\, \ud S_\eta\, \ud \tau\biggr\rvert^p \frac{1}{\abs{\xi-\xi'}^{n+p(r-\floor{r})}}\, \ud S_\xi\, \ud S_{\xi'}\biggr)^{q/p}\ud t\Biggr)^{1/q}.
    \end{multline*}
	Next we introduce a partition of unity $\sum_{\alpha}\psi^{\alpha}$ on $S$
    \begin{multline*}
    		\Biggl(\int_0^T\!\!\!\biggl(\int_S\int_S \Big\lvert\sum_\alpha\int_0^t\!\!\!\int_S N_l(\xi-\eta,t-\tau)\Ud^{\floor{r}}_\eta\big(\psi^{\alpha}(\eta)g_l(\eta,\tau)\big)\, \ud S_\eta\, \ud \tau \\
    		- \sum_{\alpha}\int_0^t\!\!\!\int_SN_l(\xi'-\eta,t-\tau)\Ud^{\floor{r}}_{\eta}\big(\psi^{\alpha}(\eta)g_l(\eta,\tau)\big)\, \ud S_\eta\, \ud \tau\Big\rvert^p \\
    		\cdot \frac{1}{\abs{\xi-\xi'}^{n+p(r-\floor{r})}}\, \ud S_\xi\, \ud S_{\xi'}\biggr)^{q/p}\ud t\Biggr)^{1/q}
	\end{multline*}
	and integrate by parts
	\begin{multline*}
		\Biggl(\int_0^T\!\!\!\biggl(\int_S\int_S \Big\lvert\sum_\alpha\int_0^t\!\!\!\int_S N_l(\eta,\tau)\Ud^{\floor{r}}_\xi\big(\psi^{\alpha}(\xi - \eta)g_l(\xi - \eta,t - \tau)\big)\, \ud S_\eta\, \ud \tau \\
        - \sum_{\alpha}\int_0^t\!\!\!\int_SN_l(\eta,\tau)\Ud^{\floor{r}}_{\xi'}\big(\psi^{\alpha}(\xi' - \eta)g_l(\xi' - \eta,t - \tau)\big)\, \ud S_\eta\, \ud \tau\Big\rvert^p \\
        \cdot \frac{1}{\abs{\xi-\xi'}^{n+p(r-\floor{r})}}\, \ud S_\xi\, \ud S_{\xi'}\biggr)^{q/p}\ud t\Biggr)^{1/q}.
    \end{multline*}
    Applying now the general Minkowski inequality and Lemma \ref{l9} yields
    \begin{equation*}
        \left(\int_0^T\!\!\!\left(\int_S\int_S \frac{\abs{\Ud^{\floor{r}}_\xi\! N_lg_l(\xi,t) - \Ud^{\floor{r}}_{\xi'}\! N_lg_l(\xi',t)}^p}{\abs{\xi-\xi'}^{n+p(r-\floor{r})}}\, \ud S_\xi\, \ud S_{\xi'}\right)^{q/p}\ud t\right)^{1/q} \leq \abs{N_l} \abs{\Psi} \abs{S} t.
    \end{equation*}

    Let us now assume that for $k$ given \eqref{teza2} holds. We will show that it is true also for $k+1$. We see
    \begin{multline*}
    		\left(\int_0^T\!\!\!\left(\int_S\int_S \frac{\abs{\Ud^{\floor{r}}_\xi\! N^{k+1}_lg_l(\xi,t)-\Ud^{\floor{r}}_{\xi'}\! N^{k+1}_lg_l(\xi',t)}^p}{\abs{\xi-\xi'}^{n+p(r-\floor{r})}}\, \ud S_\xi\, \ud S_{\xi'}\right)^{q/p}\ud t\right)^{1/q}\\
    		= \Biggl(\int_0^T\!\!\!\biggl(\int_S\int_S \biggl\lvert\int_0^t\!\!\!\int_S \Ud^{\floor{r}}_\xi\! N_l(\xi-\eta,t-\tau)N^k_lg_l(\eta,\tau)\, \ud S_\eta\, \ud \tau - \int_0^t\!\!\!\int_S\Ud^{\floor{r}}_{\xi'}\! N_l(\xi'-\eta,t-\tau)N^k_lg_l(\eta,\tau)\, \ud S_\eta\, \ud \tau\biggr\rvert^p\\
    		\cdot \frac{1}{\abs{\xi-\xi'}^{n+p(r-\floor{r})}}\, \ud S_\xi\, \ud S_{\xi'}\biggr)^{q/p}\ud t\Biggr)^{1/q}.
    \end{multline*}

    Applying a partition of unity $\sum_\alpha \psi^\alpha$, integrating by parts and using the general Minkowski inequality and Lemma \ref{l9}, as we did in case $k = 1$, we conclude the proof.
\end{proof}

\begin{rem}\label{czw_wyr}
	Let $\abs{N_l}$ be defined as in \eqref{N} and let
	\begin{equation*}
		\abs{\Theta} := \left(\int_0^T\!\!\!\int_0^T\!\!\!\frac{\left(\int_S \abs{\partial^{\floor{s}}_t\! g_l(\xi,t) - \partial^{\floor{s}}_w\! g_l(\xi,w)}^p\, \ud x\right)^{q/p}}{\abs{t - w}^{1+q(s-\floor{s})}}\, \ud t\, \ud w\right)^{1/q} + \abs{\Phi_{0,\floor{s}}}.
	\end{equation*}
	Then repeating the proof of the last Lemma we get
	\begin{equation}\label{teza3}
		\left(\int_0^T\!\!\!\int_0^T\!\!\!\left(\int_S \frac{\abs{\Ud^{\floor{s}}_t\! N^k_lg_l(\xi,t)-\Ud^{\floor{s}}_w\! N^k_lg_l(\xi,w)}^p}{\abs{t-w}^{1+q(s-\floor{s})}}\, \ud S_\xi\right)^{q/p}\ud t\, \ud w\right)^{1/q} \leq \abs{\Theta} \abs{N_l}^k \abs{S}^k \frac{T^k}{k!}.
	\end{equation}
\end{rem}

It remains to check that $\norm{g_l}_{W^{r,s}_{p,q}(S^T)} \leq c \norm{\varphi}_{W^{r,s}_{p,q}(S^T)}$. This can be easily seen if we consider the second equation in \eqref{r23}, apply any partition of unity on $S$ and repeat the calculations from this section for each of four terms of the norm $W^{r,s}_{p,q}(S^T)$.

\section{Estimates in the half space}\label{s4} 

In this section we estimate the solution $u$ to the problem \eqref{p1} in the half space. 

\begin{proof}[Proof of Theorem \ref{t2}]
    The proof is divided into four lemmas. Point \eqref{i1} is proved in Lemma \ref{l7} for $\alpha = 0$ and in Lemma \ref{l14} for $0 < \alpha < \frac{1}{p}$. Point \eqref{i2} is proved in Lemma \ref{l6} and point \eqref{i3} in Lemma \ref{l15}.
\end{proof}

\begin{lem}\label{l7}
    Let us assume that $\varphi \in L_q(0,T;L_p(\mathbb{R}^{n - 1}))$, where $1 \leq p < \infty$, $1 \leq q \leq \infty$. If $u$ is a solution to problem \eqref{p1}, then $u\in L_q(0,T;L_p(\mathbb{R}^{n}_+))$ and
    \begin{equation*}
        \norm{u}_{L_q(0,T;L_p(\mathbb{R}^n_+))} \leq c(p,T) \norm{\varphi}_{L_q(0,T;L_p(\mathbb{R}^{n - 1}))},
    \end{equation*}
    where the constant $c$ depends on $p$ and $T$.
\end{lem}

\begin{proof}
    From \eqref{r2} it follows that
    \begin{equation}\label{r20}
        \abs{u(x,t)} \leq \frac{1}{(4\pi)^{\frac{n}{2}}} \int_0^t\!\!\!\int_{\mathbb{R}^{n - 1}} \frac{x_n}{\tau^{\frac{n + 2}{2}}} e^{-\frac{\abs{y'}^2 + x_n^2}{4\tau)}}\abs{\varphi(x' - y',t - \tau)}\, \ud y'\, \ud \tau.
    \end{equation}
    Taking the $L_q\big(0,T;L_p(\mathbb{R}^n_+)\big)$-norm yields
    \begin{multline*}
        \norm{u}_{L_q(0,T;L_p(\mathbb{R}^n_+))} \leq \frac{1}{(4\pi)^{\frac{n}{2}}} \bigg(\int_0^T\bigg(\int_{\mathbb{R}^n_+}\bigg\lvert\int_0^t\!\!\!\int_{\mathbb{R}^{n - 1}} \frac{x_n}{\tau^{\frac{n + 2}{2}}} e^{-\frac{\abs{y'}^2 + x_n^2}{4\tau}} \\
        \cdot \abs{\varphi(x' - y',t - \tau)}\, \ud y'\, \ud \tau\bigg\rvert^p\, \ud x\bigg)^{\frac{q}{p}}\, \ud t\bigg)^{\frac{1}{q}} \\
        \leq \frac{1}{(4\pi)^{\frac{n}{2}}} \bigg(\int_{\mathbb{R}_+}\bigg\lvert\int_0^T\!\!\!\int_{\mathbb{R}^{n - 1}} \frac{x_n}{\tau^{\frac{n + 2}{2}}} e^{-\frac{\abs{y'}^2 + x_n^2}{4\tau}}\, \ud y'\, \ud \tau \bigg\rvert^p\, \ud x_n\bigg)^{\frac{1}{p}} \norm{\varphi}_{L_q(0,T;L_p(\mathbb{R}^{n - 1}))},
    \end{multline*}
    where in the last inequality we used the general Minkowski inequality (Lemma \ref{l3}) with respect to $x'$. The integral with respect to $y'$ is equal to $(4\pi)^{\frac{n - 1}{2}} \tau^{\frac{n - 1}{2}}$, so using the general Minkowski inequality with respect to $x_n$ gives
    \begin{multline*}
        \norm{u}_{L_q(0,T;L_p(\mathbb{R}^n_+))} \leq \frac{1}{\sqrt{4\pi}} \int_0^T\frac{1}{\tau^{\frac{3}{2}}}\bigg(\int_0^\infty x_n^p e^{-\frac{p x_n^2}{4\tau}}\, \ud x_n\bigg)^{\frac{1}{p}}\, \ud \tau \norm{\varphi}_{L_q(0,T;L_p(\mathbb{R}^{n - 1}))} \\
        = \frac{(\Gamma(\frac{1 + p}{2}))^{\frac{1}{p}}}{\sqrt{\pi}p^{\frac{1}{2} + \frac{1}{2p}}} \int_0^T\frac{\tau^{\frac{1}{2p} + \frac{1}{2}}}{\tau^{\frac{3}{2}}} \, \ud \tau \norm{\varphi}_{L_q(0,T;L_p(\mathbb{R}^{n - 1}))} \leq c(p,T) \norm{\varphi}_{L_q(0,T;L_p(\mathbb{R}^{n - 1}))}
    \end{multline*}
    for any $1 \leq p < \infty$. This concludes the proof.
\end{proof}

\begin{lem}\label{l14}
    Suppose that $\varphi \in L_q(0,T;L_p(\mathbb{R}^{n - 1}))$, where $1 \leq p < \infty$, $1 \leq q \leq \infty$. If $u$ is a solution to problem \eqref{p1}, then $u\in L_q(0,T;W^\alpha_p(\mathbb{R}^{n}_+))$ and
    \begin{equation*}
        \norm{u}_{L_q(0,T;W^\alpha_p(\mathbb{R}^n_+))} \leq c(p,T) \norm{\varphi}_{L_q(0,T;L_p(\mathbb{R}^{n - 1}))},
    \end{equation*}
    where the constant $c$ depends on $p$ and $T$ and $\alpha$ is such that $0 < \alpha < \frac{1}{p}$.
\end{lem}

\begin{proof}
    We write
    \begin{equation*}
        u(x,t) - u(z,t) = u(x',x_n,t) - u(x',z_n,t) + u(x',z_n,t) - u(z',z_n,t) =:I_1 + I_2.
    \end{equation*}
    Let us first consider $I_1$. From Lemma \ref{l5} it follows that
    \begin{multline*}
        \frac{u(x',x_n,t) - u(x',z_n,t)}{\abs{x - z}^{\frac{n}{p} + \alpha}} \\
        = \frac{1}{(4\pi)^{\frac{n}{2}}} \int_0^t\!\!\!\int_{\mathbb{R}^{n - 1}} \frac{\left(\frac{x_n}{\tau^{\frac{n + 2}{2}}}e^{-\frac{\abs{y'}^2 + x_n^2}{4\tau}} - \frac{z_n}{\tau^{\frac{n + 2}{2}}}e^{-\frac{\abs{y'}^2 + z_n^2}{4\tau}}\right)\varphi(x' - y', t - \tau)}{\abs{x - z}^{\frac{n}{p} + \alpha}}\, \ud y'\, \ud \tau.
    \end{multline*}
    Taking the $L_p$-norm with respect to $x'$ and $z'$ and applying the general Minkowski inequality yields
    \begin{multline*}
        \left(\int_{\mathbb{R}^{n - 1}}\int_{\mathbb{R}^{n - 1}}\frac{\abs{u(x',x_n,t) - u(x',z_n,t)}^p}{\abs{x - z}^{n + \alpha p}}\, \ud x'\, \ud z'\right)^{\frac{1}{p}} \\
        \leq \frac{1}{(4\pi)^{\frac{n}{2}}} \int_0^t\!\!\!\int_{\mathbb{R}^{n - 1}} \frac{e^{-\frac{\abs{y'}^2}{4\tau}}\left\lvert\frac{x_n}{\tau^{\frac{n + 2}{2}}}e^{-\frac{x_n^2}{4\tau}} - \frac{z_n}{\tau^{\frac{n + 2}{2}}}e^{-\frac{z_n^2}{4\tau}}\right\rvert\norm{\varphi(t - \tau)}_{L_p(\mathbb{R}^{n - 1})}}{\abs{x_n - z_n}^{\frac{1}{p} + \alpha}}\, \ud y'\, \ud \tau \\
        = \frac{1}{(4\pi)^{\frac{1}{2}}} \int_0^t\frac{\left\lvert\frac{x_n}{\tau^{\frac{3}{2}}}e^{-\frac{x_n^2}{4\tau}} - \frac{z_n}{\tau^{\frac{3}{2}}}e^{-\frac{z_n^2}{4\tau}}\right\rvert}{\abs{x_n - z_n}^{\frac{1}{p} + \alpha}}\norm{\varphi(t - \tau)}_{L_p(\mathbb{R}^{n - 1})}\, \ud \tau.
    \end{multline*}
    Next we take the $L_p$-norm with respect to $x_n$ and $z_n$ and apply the general Minkowski inequality
    \begin{multline*}
        \left(\int_{\mathbb{R}^n_+}\int_{\mathbb{R}^n_+}\frac{\abs{u(x',x_n,t) - u(x',z_n,t)}^p}{\abs{x - z}^{n + \alpha p}}\, \ud x\, \ud z\right)^{\frac{1}{p}} \\
        \leq \frac{1}{(4\pi)^{\frac{1}{2}}} \int_0^t\left(\int_0^\infty\!\!\!\int_0^\infty\frac{\left\lvert\frac{x_n}{\tau^{\frac{3}{2}}}e^{-\frac{x_n^2}{4\tau}} - \frac{z_n}{\tau^{\frac{3}{2}}}e^{-\frac{z_n^2}{4\tau}}\right\rvert^p}{\abs{x_n - z_n}^{1 + \alpha p}}\, \ud x_n\, \ud z_n\right)^{\frac{1}{p}} \norm{\varphi(t - \tau)}_{L_p(\mathbb{R}^{n - 1})}\, \ud \tau \\
        = \frac{1}{(4\pi)^{\frac{1}{2}}} \int_0^t\frac{2^{\frac{1}{p}}}{\tau^{1 + \frac{\alpha}{2} - \frac{1}{2p}}}\left(\int_0^\infty\!\!\!\int_0^\infty\frac{\left\lvert\bar x_n e^{-\bar x_n^2} - \bar z_n e^{-\bar z_n^2}\right\rvert^p}{\abs{\bar x_n - \bar z_n}^{1 + \alpha p}}\, \ud \bar x_n\, \ud \bar z_n\right)^{\frac{1}{p}} \norm{\varphi(t - \tau)}_{L_p(\mathbb{R}^{n - 1})}\, \ud \tau \\
= c(p) \int_0^t \frac{1}{\tau^{1 + \frac{\alpha}{2} - \frac{1}{2p}}} \norm{\varphi(t - \tau)}_{L_p(\mathbb{R}^{n - 1})}\, \ud \tau.
    \end{multline*}
    Finally we take the $L_q$-norm and apply the general Minkowski inequality. This shows that
    \begin{equation}\label{eq49}
    \begin{aligned}
        \norm{I_1}_{L_q(0,T;W^{\alpha}_p(\mathbb{R}^n_+))} \leq& c(p) \int_0^T \frac{1}{\tau^{1 + \frac{\alpha}{2} - \frac{1}{2p}}}\, \ud \tau \norm{\varphi}_{L_q(0,T;L_p(\mathbb{R}^{n - 1}))} \\
        \leq& c(p,T) \norm{\varphi}_{L_q(0,T;L_p(\mathbb{R}^{n - 1}))}
    \end{aligned}
    \end{equation}
    if and only if $\alpha < \frac{1}{p}$.

    Next, let us consider $I_2$. Dividing by $\abs{x' - z'}^{\frac{n}{p} + \alpha}$ and applying the H\"older inequality with respect to $y'$ leads to
    \begin{multline*}
        \frac{1}{(4\pi)^{\frac{n}{2}}}\int_0^t \frac{z_n}{\tau^{\frac{n + 2}{2}}}e^{-\frac{z_n^2}{4\tau}}\underbrace{\left(\int_{\mathbb{R}^{n - 1}} \abs{e^{-\frac{\abs{x' - y'}^2}{4\tau}} - e^{-\frac{\abs{z' - y'}^2}{4\tau}}}^{\frac{p'}{2}}\, \ud y'\right)^{\frac{1}{p'}}}_{J_1} \\
        \cdot \left(\int_{\mathbb{R}^{n - 1}}\frac{\abs{e^{-\frac{\abs{x' - y'}^2}{4\tau}} - e^{-\frac{\abs{z' - y'}^2}{4\tau}}}^{\frac{p}{2}}}{\abs{x - z}^{n + \alpha p}}\abs{\varphi(y',t - \tau)}^p\, \ud y'\right)^{\frac{1}{p}}\, \ud \tau =: \bar I_2.
    \end{multline*}
    Let us calculate $J_1$. From the Mean Value Theorem we have that, for some $\theta \in (0,1)$
    \begin{multline*}
        J_1 = \left(\int_{\mathbb{R}^{n - 1}} \abs{\frac{\abs{x' - z'}}{2\tau}e^{-\frac{\abs{x' - (y' + \theta(x' - z'))}^2}{4\tau}}\abs{x' - (y' + \theta(x' - z'))}}^{\frac{p'}{2}}\, \ud y'\right)^{\frac{1}{p'}} \\
        \leq c(n,p) \frac{\tau^{\frac{n - 1}{2p'}}}{\tau^{\frac{1}{4}}}\abs{x' - z'}^{\frac{1}{2}}.
    \end{multline*}
    Taking the $L_p$-norm with respect to $x$ and $z$ we get
    \begin{multline*}
        \left(\int_{\mathbb{R}^n_+}\int_{\mathbb{R}^n_+} \abs{\bar I_2}^p\, \ud x\, \ud y\right)^{\frac{1}{p}} \leq c(n,p) \Bigg(\int_{\mathbb{R}^n_+}\int_{\mathbb{R}^n_+}\Bigg\lvert\int_0^t \frac{z_n}{\tau^{\frac{n + 2}{2}}}e^{-\frac{z_n^2}{4\tau}} \frac{\tau^{\frac{n - 1}{2p'}}}{\tau^{\frac{1}{4}}}\abs{x' - z'}^{\frac{1}{2}} \\
        \cdot \left(\int_{\mathbb{R}^{n - 1}}\frac{\abs{e^{-\frac{\abs{x' - y'}^2}{4\tau}} - e^{-\frac{\abs{z' - y'}^2}{4\tau}}}^{\frac{p}{2}}}{\abs{x - z}^{n + \alpha p}}\abs{\varphi(y',t - \tau)}^p\, \ud y'\right)^{\frac{1}{p}}\, \ud \tau\Bigg\rvert^p\, \ud x\, \ud y\Bigg)^{\frac{1}{p}} \\
        \leq c(n,p)\int_0^t \frac{1}{\tau^{\frac{n + 1}{2}}} \frac{\tau^{\frac{n - 1}{2p'}}}{\tau^{\frac{1}{4}}} \\
        \cdot \left(\int_{\mathbb{R}^{n - 1}}\int_{\mathbb{R}^n_+}\int_{\mathbb{R}^{n - 1}}e^{-\frac{z_n^2p}{4\tau}}\frac{\abs{e^{-\frac{\abs{x' - y'}^2}{4\tau}} - e^{-\frac{\abs{z' - y'}^2}{4\tau}}}^{\frac{p}{2}}}{\abs{x' - z'}^{n - 1 + p(\alpha -\frac{1}{2})}}\abs{\varphi(y',t - \tau)}^p\, \ud x'\, \ud z\, \ud y'\right)^{\frac{1}{p}}\, \ud \tau,
    \end{multline*}
    where in the last inequality we used Lemma \ref{l2}, the General Minkowski inequality and we integrated with respect to $x_n$. Next we change the variables
    \begin{align*}
        \frac{x' - y'}{2\sqrt{\tau}} = \bar x' & & \ud x' = 2^{n - 1}\tau^{\frac{n - 1}{2}}\, \ud \bar x' & & \frac{z' - y'}{2\sqrt{\tau}} = \bar z' & & \ud z' = 2^{n - 1}\tau^{\frac{n - 1}{2}}\, \ud \bar z',
    \end{align*}
    hence
    \begin{multline*}
        \left(\int_{\mathbb{R}^n_+}\int_{\mathbb{R}^n_+} \abs{\bar I_2}^p\, \ud x\, \ud y\right)^{\frac{1}{p}} \leq c(n,p)\int_0^t \frac{1}{\tau^{\frac{n + 1}{2}}} \frac{\tau^{\frac{n - 1}{2p'}}}{\tau^{\frac{1}{4}}} \frac{\tau^{\frac{n - 1}{p}}}{\tau^{\frac{1}{2p}(n - 1 + p(\alpha - \frac{1}{2}))}}\\
        \cdot \left(\int_{\mathbb{R}^n_+}\int_{\mathbb{R}^{n - 1}}e^{-\frac{z_n^2p}{4\tau}}\frac{\abs{e^{-\abs{\bar x'}^2} - e^{{\abs{\bar z'}^2}}}^{\frac{p}{2}}}{\abs{\bar x' - \bar z'}^{n - 1 + p(\alpha -\frac{1}{2})}}\, \ud x'\, \ud z\right)^{\frac{1}{p}}\norm{\varphi(t - \tau)}_{L_p(\mathbb{R}^{n - 1})}\, \ud \tau \\
        = c(n,p) \int_0^t \frac{\tau^{\frac{1}{2p}}}{\tau^{\frac{n + 1}{2}}} \frac{\tau^{\frac{n - 1}{2p'}}}{\tau^{\frac{1}{4}}} \frac{\tau^{\frac{n - 1}{p}}}{\tau^{\frac{1}{2p}(n - 1 + p(\alpha - \frac{1}{2}))}} \norm{\varphi(t - \tau)}_{L_p(\mathbb{R}^{n - 1})}\, \ud \tau \\
        = c(n,p) \int_0^t \frac{1}{\tau^{1 + \frac{\alpha}{2} - \frac{1}{2p}}} \norm{\varphi(t - \tau)}_{L_p(\mathbb{R}^{n - 1})}\, \ud \tau.
    \end{multline*}
    Next we apply the $L_q$-norm and use the general Minkowski inequality. It yields
    \begin{multline*}
        \left(\int_0^T\left(\int_{\mathbb{R}^n_+}\int_{\mathbb{R}^n_+}\frac{\abs{u(x',z_n,t) - u(z',z_n,t)}^p}{\abs{x - z}^{n + \alpha p}}\, \ud x\, \ud z\right)^{\frac{q}{p}}\, \ud t\right)^{\frac{1}{q}} \\
        \leq c(n,p) \norm{\varphi}_{L_q(0,T;L_p(\mathbb{R}^{n - 1}))} \int_0^T \frac{1}{\tau^{1 + \frac{\alpha}{2} - \frac{1}{2p}}} \, \ud \tau \leq c(n,p,T) \norm{\varphi}_{L_q(0,T;L_p(\mathbb{R}^{n - 1}))}
    \end{multline*}
    if and only if $\alpha < \frac{1}{p}$. From \eqref{eq49} and the inequality above we conclude the proof.
\end{proof}

\begin{lem}\label{l6}
    Let us assume that $\varphi \in L_q(0,T;W_p^{1 - \frac{1}{p}}(\mathbb{R}^{n - 1}))$, where $1 \leq p < \infty$, $1 \leq q \leq \infty$. Then $\Ud_{x'} u \in L_q(0,T;L_p(\mathbb{R}^n_+))$ and
    \begin{equation*}
        \norm{\Ud_{x'} u(x,t)}_{L_q(0,T;L_p(\mathbb{R}^n_+))} \leq c(n,p) \norm{\varphi}_{L_q(0,T;W^{1 - \frac{1}{p}}_p(\mathbb{R}^{n - 1}))}
    \end{equation*}
    holds.
\end{lem}

\begin{proof}
    We follow and generalize \cite[Sec. 4, $\S$3]{lad}. In view of \eqref{r2} and Lemma \ref{l4} we have
    \begin{equation*}
        \Ud_{x'} u(x,t) = -2 \int_0^{\infty}\!\!\!\int_{\mathbb{R}^{n - 1}} \Ud_{y'}\partial_{x_n}\Gamma(y',x_n,\tau)\big(\varphi(x' - y',t - \tau) - \varphi(x',t - \tau)\big)\, \ud y'\, \ud \tau.
     \end{equation*}
     Using Lemma \ref{l2}, the H\"older inequality with respect to $x'$ and the general Minkowski inequality yields
     \begin{multline*}
         \norm{\Ud_{x'} u(x,t)}_{L_p(\mathbb{R}^{n - 1})} \leq c(n)\int_0^{\infty} \frac{1}{\tau^{\frac{n + 2}{2}}}\int_{\mathbb{R}^{n - 1}}e^{-\frac{\abs{y'}^2 + x_n^2}{16\tau}}\\
         \cdot \left(\int_{\mathbb{R}^{n - 1}}\abs{\mu(x' - y',t - \tau) - \mu(x',t - \tau)}^p\, \ud x'\right)^{\frac{1}{p}}\,\ud y\, \ud \tau \\
         = c(n)\int_{\mathbb{R}^{n - 1}} \frac{1}{(\abs{y'}^2 + x_n^2)^{\frac{n}{2}}} \left(\int_{\mathbb{R}^{n - 1}}\abs{\varphi(x' - y',t - \tau) - \varphi(x',t - \tau)}^p\, \ud x'\right)^{\frac{1}{p}}\,\ud y.
     \end{multline*}
     Applying the H\"older inequality with respect to $y'$ gives
     \begin{multline*}
         \norm{\Ud_{x'} u(x,t)}_{L_p(\mathbb{R}^{n - 1})} \leq c(n) \left(\int_{\mathbb{R}^{n - 1}}\frac{\int_{\mathbb{R}^{n - 1}}\abs{\varphi(x' - y',t) - \varphi(x',t)}^p\, \ud x'}{(\abs{y'}^2 + x_n^2)^{\frac{n - 1}{2} + \frac{p}{2} - \frac{1}{4}}}\, \ud y'\right)^{\frac{1}{p}} \\
         \cdot \left(\int_{\mathbb{R}^{n - 1}} \frac{1}{(\abs{y'}^2 + x_n^2)^{\frac{n - 1}{2} + \frac{p'}{4p}}}\, \ud y'\right)^{\frac{1}{p'}} \\
         = c(n,p) \frac{1}{x_n^{\frac{1}{2p}}} \left(\int_{\mathbb{R}^{n - 1}}\frac{\int_{\mathbb{R}^{n - 1}}\abs{\varphi(x' - y',t) - \varphi(x',t)}^p\, \ud x'}{(\abs{y'}^2 + x_n^2)^{\frac{n - 1}{2} + \frac{p}{2} - \frac{1}{4}}}\, \ud y'\right)^{\frac{1}{p}}.
     \end{multline*}
     Next we take the $L_p$ norm with respect to $x_n$
     \begin{equation*}
         \norm{\Ud_{x'} u(x,t)}_{L_p(\mathbb{R}_+^n)} \leq c(n,p) \left(\int_{\mathbb{R}^{n - 1}}\frac{\int_{\mathbb{R}^{n - 1}}\abs{\varphi(x' - y',t) - \varphi(x',t)}^p\, \ud x'}{\abs{y'}^{n - 2 + p}}\, \ud y'\right)^{\frac{1}{p}},
     \end{equation*}
     and he $L_q$ norm with respect to $t$. This ends the proof.
\end{proof}

\begin{lem}\label{l15}
    Suppose that $\varphi(x',t)\in W^{1 - \frac{1}{p}, \frac{1}{2} - \frac{1}{2p}}_{p,q}(\mathbb{R}^{n - 1}\times(0,T))$, where $1 \leq q \leq p < \infty$. Then $\partial_{x_n} u(x,t) \in L_q(0,T;L_p(\mathbb{R}^n_+))$ and
    \begin{equation*}
         \norm{\partial_{x_n}u}_{L_q(0,T;L_p(\mathbb{R}^n_+))} \leq c(n,p,T)\norm{\varphi}_{W_{p,q}^{1 - \frac{1}{p}, \frac{1}{2} - \frac{1}{2p}}(\mathbb{R}^{n - 1}\times (0,T))}
    \end{equation*}
    holds.
\end{lem}

\begin{proof}
    In view of \eqref{r2} and Lemma \ref{l4} we have
    \begin{multline*}
            \partial_{x_n}u(x,t) = \frac{1}{(4\pi)^{\frac{n}{2}}} \int_0^{\infty}\!\!\!\int_{\mathbb{R}^{n - 1}} \frac{1}{\tau^{\frac{n}{2}}}\partial^2_{x_nx_n}e^{-\frac{\abs{y'}^2 + x_n^2}{4\tau}}\varphi(x' - y',t - \tau)\, \ud y'\, \ud \tau \\
            = \frac{1}{(4\pi)^{\frac{n}{2}}} \int_0^{\infty}\!\!\!\int_{\mathbb{R}^{n - 1}} \frac{1}{\tau^{\frac{n}{2}}}\partial_{\tau}e^{-\frac{\abs{y'}^2 + x_n^2}{4\tau}}\big(\varphi(x' - y',t - \tau) - \varphi(x',t - \tau)\big)\, \ud y'\, \ud \tau \\
            - \frac{1}{(4\pi)^{\frac{n}{2}}} \int_0^{\infty}\!\!\!\int_{\mathbb{R}^{n - 1}} \frac{1}{\tau^{\frac{n}{2}}}\Ud^2_{y'y'}e^{-\frac{\abs{y'}^2 + x_n^2}{4\tau}}\big(\varphi(x' - y',t - \tau) - \varphi(x' - y',t)\big)\, \ud y'\, \ud \tau\\
             =: I_1 + I_2.
     \end{multline*}

     Let us consider $I_1$. Taking the absolute value of both sides and using Lemma \ref{l2} we get
     \begin{equation*}
         \abs{I_1} \leq c\int_0^{\infty}\frac{1}{\tau^{\frac{n + 2}{2}}}\int_{\mathbb{R}^{n - 1}}e^{-\frac{\abs{y'}^2 + x_n^2}{16\tau}} \abs{\varphi(x' - y',t - \tau) - \varphi(x',t - \tau)}\,\ud y\, \ud \tau.
     \end{equation*}
     Next we repeat the calculations from the proof of Lemma \ref{l6}. This yields
     \begin{equation}\label{r34}
         \norm{I_1}_{L_q(0,T;L_p(\mathbb{R}_+^n))} \leq c(n,p) \norm{\varphi}_{L_q(0,T;W^{1 - \frac{1}{p}}_p(\mathbb{R}^{n - 1}))}.
     \end{equation}

     It remains to estimate $I_2$. Since
     \begin{equation*}
         \abs{I_2} \leq c\int_0^{\infty}\frac{1}{\tau^{\frac{n + 2}{2}}}\int_{\mathbb{R}^{n - 1}}e^{-\frac{\abs{y'}^2 + x_n^2}{16\tau}} \abs{\varphi(x' - y',t - \tau) - \varphi(x' - y',t)}\,\ud y\, \ud \tau,
     \end{equation*}
     which follows from Lemma \ref{l2}, we can take the $L_p$ norm with respect to $x'$ and integrate over $y'$. We obtain
     \begin{equation}\label{r29}
         \norm{I_2}_{L_p(\mathbb{R}^{n - 1})} \leq c \int_0^\infty \frac{1}{\tau^{\frac{3}{2}}} e^{-\frac{x_n^2}{16\tau}} \abs{N(t-\tau,t)}\, \ud \tau,
     \end{equation}
     where
     \begin{equation*}
         N(t - \tau,t) := \left(\int_{\mathbb{R}^{n - 1}} \abs{\varphi(x',t - \tau) - \varphi(x',t)}^p\, \ud x' \right)^{\frac{1}{p}}.
     \end{equation*}
     Next we rewrite the right-hand side in \eqref{r29} as follows
     \begin{equation*}
         c\int_0^\infty \frac{1}{\tau^{\alpha_1}} e^{-\frac{x_n^2}{32\tau}} \abs{N(t-\tau,t)} \frac{1}{\tau^{\alpha_2}} e^{-\frac{x_n^2}{32\tau}}\, \ud \tau,
     \end{equation*}
     where $\alpha_1 + \alpha_2 = \frac{3}{2}$, and apply the H\"older inequality. This yields
     \begin{multline*}
         \norm{I_2}_{L_p(\mathbb{R}^{n - 1})} \leq c \left(\int_0^{\infty}\frac{1}{\tau^{\alpha_1q}}e^{-\frac{qx_n^2}{32\tau}}\abs{N(t - \tau,t)}^q\, \ud \tau\right)^{\frac{1}{q}} \left(\int_0^{\infty}\frac{1}{\tau^{\alpha_2q'}}e^{-\frac{q'x_n^2}{32\tau}}\, \ud \tau\right)^{\frac{1}{q'}} \\
         = \frac{c}{x_n^{2\alpha_2 - \frac{2}{q'}}} \left(\int_0^{\infty}\frac{1}{\tau^{\alpha_1q}}e^{-\frac{qx_n^2}{32\tau}}\abs{N(t - \tau,t)}^q\, \ud \tau\right)^{\frac{1}{q}}.
     \end{multline*}
     Taking the $L_q$ norm with respect to $x_n$ gives
     \begin{multline*}
         \norm{I_2}_{L_p(\mathbb{R}_+^n)} \leq c \left(\int_0^{\infty}\left(\int_0^{\infty}\frac{1}{\tau^{\alpha_1q}} \frac{1}{x_n^{2\alpha_2q - \frac{2q}{q'}}} e^{-\frac{qx_n^2}{32\tau}}\abs{N(t - \tau,t)}^q\, \ud \tau\right)^{\frac{p}{q}}\, \ud x_n\right)^{\frac{1}{p}} \\
         \leq c \left(\int_0^{\infty}\left(\int_0^{\infty}\frac{1}{\tau^{\alpha_1p}}\frac{1}{x_n^{2\alpha_2p - \frac{2p}{q'}}}e^{-\frac{px_n^2}{32\tau}}\, \ud x_n\right)^{\frac{q}{p}}\abs{N(t - \tau,t)}^q\, \ud \tau\right)^{\frac{1}{q}} = c \left(\int_0^{\infty} \frac{\abs{N(t - \tau,t)}^q}{\tau^{1 + q\left(\frac{1}{2} - \frac{1}{2p}\right)}}\, \ud \tau\right)^{\frac{1}{q}}.
     \end{multline*}
     Integrating with respect to $t$ in $q$-th power and taking into account the estimate \eqref{r34} we conclude the proof.
\end{proof}

\section{Estimates in bounded domains}\label{s5}

In this section we give the proof of Theorem \ref{t3}. We start with recalling a fundamental definition.

\begin{defi}
   We will say that $S := \partial \Omega$ belongs to $\mathcal{C}^m_{loc}$, if for every $x_0 \in S$ there exist a number $r > 0$ and a function $f\colon \mathbb{R}^{n - 1}\to\mathbb{R}$, $f\in\mathcal{C}^m_c(\mathbb{R}^{n - 1})$ such that
   \begin{equation*}
       \Omega \cap B(x_0,r) = \left\{x\in B(x_0,r)\colon x_n > f(x_1,\ldots,x_{n - 1})\right\}.
   \end{equation*}
\end{defi}

\begin{proof}[Proof of Theorem \ref{t3}]
For any fixed $\lambda > 0$ (it will be chosen later) we cover $\Omega$ by sets $\Omega^{(k)}$, $\Omega \subset \bigcup_{k = 0}^{N_\lambda}\Omega^{(k)}$, which satisfy
\begin{equation*}
	\frac{1}{2}\lambda < \dist(\Omega^0,\partial \Omega) < \lambda, \qquad \Omega^{(k)} := B(x_0^{(k)},r_\lambda) \cap S.
\end{equation*}
Next we introduce a partition of unity $\sum_{k = 0}^{N_\lambda} \eta^{(k)} = 1$, which is subordinated into the covering $\Omega^{(k)}$. Then, multiplying \eqref{p111} by $\eta^{(k)}$ we obtain
\begin{equation}\begin{aligned}\label{p112}
    &\eta^{(k)}u,_t - \triangle (\eta^{(k)}u) = -2\nabla \eta^{(k)} \nabla u - u \triangle \eta^{(k)} =: g^{(k)} & &\textrm{in } \Omega^T,\\
    &\eta^{(k)} u = \eta^{(k)} \varphi & &\textrm{on } S^T,\\
    &\eta^{(k)} u\vert_{t=0} = 0 & &\textrm{in } \Omega\times\{0\}.
    \end{aligned}
\end{equation}
Let us denote $u^{(k)} := \eta^{(k)}u$. Then, for $k = 0$ the solution to problem \eqref{p112} can be expressed in the form
\begin{align*}
	u^{(0)}(x,t) &= \int_{\mathbb{R}^n\times(0,T)} G(x - y, t - \tau) \left(-2 \nabla \eta^{(0)}(y)\nabla u (y,\tau) - \triangle \eta^{(0)}(y) u(y,\tau)\right)\ud y\, \ud \tau \\
	&= \int_{\mathbb{R}^n\times(0,T)} \nabla G(x - y, t - \tau)\left(-2 \nabla \eta^{(0)}(y)u(y,\tau)\right)\ud y\, \ud \tau \\
	&- \int_{\mathbb{R}^n\times(0,T)} G(x - y,t - \tau)\left(-2 \nabla^2\eta^{(0)}(y) + \triangle \eta^{(0)}\right)u(y,\tau) \ud y\, \ud \tau.
\end{align*}
For the last integral the estimate
\begin{multline*}
	\norm{\int_{\mathbb{R}^n\times(0,T)} G(x - y,t - \tau)\left(-2 \nabla^2\eta^{(0)}(y) + \triangle \eta^{(0)}\right)u(y,\tau) \ud y\, \ud \tau}_{L_q(0,T;L_p(\mathbb{R}^n))} \\
	\leq c_{n,p,q,\lambda}\norm{u}_{L_q(0,T;L_p(\Omega))}
\end{multline*}
holds, which implies that
\begin{equation*}
	\norm{u^{(0)}}_{W^{1,\frac{1}{2}}_{p,q}(\mathbb{R}^n\times(0,T))} \leq c_{n,p,q,\lambda}\norm{u}_{L_q(0,T;L_p(\Omega))}.
\end{equation*}

For $k > 0$ we introduce a local coordinate system $y = (y_1,\ldots,y_n)$ with the center at $x_0^{(k)}$, which we obtain from $x$ through translations and rotations. Then, from assumptions on $S$ we see that $S \cap \supp \eta^{(k)}$ is described by the equation $y_n = f^{(k)}(y_1,\ldots,y_{n - 1})$, $f^{(k)}(0) = 0$, $\nabla f^{(k)} (0) = 0$. 

Next, we straighten the set $\Omega \cap \supp \eta^{(k)}$ into the half-space through the mapping $\Phi$, $z = \Phi(y)$, where
\begin{align*}
	z' &= y', \\
	z_n &= y_n - f^{(k)}(y').	
\end{align*}
Then, \eqref{p112} in $z$-coordinates takes the form
\begin{equation}\begin{aligned}\label{p113}
    &u_{,t}^{(k)} - \triangle u^{(k)} = \left(\triangle_{\Phi^{-1}(z)} - \triangle\right)u^{(k)} + g^{(k)} & &\textrm{in } \mathbb{R}^n_+\times(0,T),\\
    &u^{(k)}\vert_{z_n = 0} = \varphi^{(k)} & &\textrm{on } \mathbb{R}^{n - 1}\times(0,T),\\
    &u^{(k)}\vert_{t=0} = 0 & &\textrm{in } \mathbb{R}^n_+\times\{0\}.
    \end{aligned}
\end{equation}
Consider now the problem
\begin{equation}\begin{aligned}\label{p114}
    &v_{,t}^{(k)} - \triangle v^{(k)} = 0 & &\textrm{in } \mathbb{R}^n_+\times(0,T),\\
    &v^{(k)}\vert_{z_n = 0} = \varphi^{(k)} & &\textrm{on } \mathbb{R}^{n - 1}\times(0,T),\\
    &v^{(k)}\vert_{t = 0} = 0 & &\textrm{in } \mathbb{R}^n_+\times\{0\}.
    \end{aligned}
\end{equation}
Then, the function $w^{(k)} = u^{(k)} - v^{(k)}$ satisfies
\begin{equation}\begin{aligned}\label{p115}
    &w_{,t}^{(k)} - \triangle w^{(k)} = \left(\triangle_{\Phi^{-1}(z)} - \triangle\right)u^{(k)} + g^{(k)} & &\textrm{in } \mathbb{R}^n_+\times(0,T),\\
    &w^{(k)}\vert_{z_n = 0} = 0 & &\textrm{on } \mathbb{R}^{n - 1}\times(0,T),\\
    &w^{(k)}\vert_{t=0} = 0 & &\textrm{in } \mathbb{R}^n_+\times\{0\}.
    \end{aligned}
\end{equation}
Now observe that
\begin{align*}
	\triangle_{\Phi^{-1}(z)} - \triangle &= \frac{\partial \Phi}{\partial y}\bigg\vert_{y = \Phi^{-1}(z)}\nabla_z \cdot \left(\frac{\partial \Phi}{\partial y}\bigg\vert_{y = \Phi^{-1}(z)}\nabla_z\right) - \nabla \cdot \nabla \\
	&= \left(\left(\frac{\partial \Phi}{\partial y}\bigg\vert_{y = \Phi^{-1}(z)}\right)^2 - 1\right)\triangle + \frac{\partial \Phi}{\partial y}\left(\frac{\partial y}{\partial z}\frac{\partial^2 \Phi}{\partial y^2}\right)\bigg\vert_{y = \Phi^{-1}(z)} \nabla
\end{align*}
Next we use the Green representation formula and integrate by parts the term containing the Laplacian. The integration is justified since the Green function vanishes on the boundary (see \cite{krz} for details). Finally, the solution to \eqref{p115} takes the form
\begin{multline}\label{eq46}
	w^{(k)}(s,t) = \int_{\mathbb{R}^n_+\times(0,T)} \nabla G(s - z, t - \tau) \left(\left(\frac{\partial \Phi}{\partial y}\bigg\vert_{y = \Phi^{-1}(z)}\right)^2 - 1\right) \nabla u^{(k)}\ud s\, \ud \tau \\
	+ \int_{\mathbb{R}^n_+\times(0,T)} G(s - z, t - \tau)\left(\nabla \left(\left(\frac{\partial \Phi}{\partial y}\bigg\vert_{y = \Phi^{-1}(z)}\right)^2 - 1\right) + \frac{\partial \Phi}{\partial y}\left(\frac{\partial y}{\partial z}\frac{\partial^2 \Phi}{\partial y^2}\right)\bigg\vert_{y = \Phi^{-1}(z)}\right)\nabla u^{(k)} \ud s\, \ud \tau \\
	+ \int_{\mathbb{R}^n_+\times(0,T)} G(s - z, t - \tau)g^{(k)}(s,\tau)\, \ud s\, \ud \tau.
\end{multline}
The first two integrals contain the small parameter $\lambda$. Let us examine the third integral. In $y$-coordinate it has the form
\begin{equation}\label{eq50}
	\int_0^T\!\!\!\int_{\Omega\cap \supp\eta^{(k)}} G(s - y, t - \tau) \left(-2 \nabla \eta^{(0)}(s)\nabla u (s,\tau) - \triangle \eta^{(0)}(s) u(s,\tau)\right)\ud s\, \ud \tau.
\end{equation}
We see that near the boundary we may write $\eta^{(k)}(s) = \eta^{(k)}_{tan}(s)\eta^{(k)}_{nor}(s)$, where the subscripts $tan$ and $nor$ denote the tangent and the normal parts. Since $\eta_{nor}^{(k)}(s) \equiv 1$ for every $k$, we infer that \eqref{eq50} does not involve differentiation along the normal direction. Hence we can integrate by parts which leads to 
\begin{multline*}
	-\int_0^T\!\!\!\int_{\Omega\cap \supp\eta^{(k)}} \nabla_{tan} G(s - y, t - \tau)\left(-2 \nabla_{tan} \eta^{(k)}(s) u(s,\tau)\right) \ud s\, \ud \tau \\
	+ \int_0^T\!\!\!\int_{\Omega\cap \supp\eta^{(k)}}G(s - y, t - \tau)\left(-2 \nabla^2_{tan} \eta^{(k)}(s) - \triangle \eta^{(k)}(s)\right) u(s,\tau)\ud s\, \ud \tau.
\end{multline*}
Passing to $z$-coordinate does not introduce tangent derivative, so for the representation \eqref{eq46} we have the estimate
\begin{equation}\label{eq45}
	\norm{w^{(k)}}_{W^{1,\frac{1}{2}}_{p,q}(\mathbb{R}^n_+\times(0,T))} \leq c_{n,p,q}\lambda\norm{u^{(k)}}_{L_q(0,T;W^1_p(\mathbb{R}^n_+))} + c_{n,p,q,\lambda}\norm{u}_{L_q(0,T;L_p(\Omega))}.
\end{equation}
Taking into account Theorem \ref{t2} and \eqref{eq45} we get instantly that
\begin{multline*}
	\norm{u^{(k)}}_{L_q(0,T;W^1_p(\mathbb{R}^n_+))} \leq \norm{u^{(k)}}_{W^{1,\frac{1}{2}}_{p,q}(\mathbb{R}^n_+\times(0,T))} \\
	\leq c_{n,p,q}\lambda\norm{u^{(k)}}_{L_q(0,T;W^1_p(\mathbb{R}^n_+))} + c_{n,p,q,\lambda}\norm{u}_{L_q(0,T;L_p(\Omega))} + \norm{\varphi^{(k)}}_{W^{1 - \frac{1}{p}, \frac{1}{2} - \frac{1}{2p}}_{p,q}(\mathbb{R}^{n - 1}\times(0,T))}.
\end{multline*}
Choosing $\lambda$ to small enough and summing over $k$ yields
\begin{equation*}
	\norm{u}_{L_q(0,T;W^1_p(\Omega))} \leq c_{n,p,q,\Omega} \norm{u}_{L_q(0,T;L_p(\Omega))} + c_{n,p,q,\Omega}\norm{\varphi}_{W^{1 - \frac{1}{p}, \frac{1}{2} - \frac{1}{2p}}_{p,q}(S^T)}.
\end{equation*}
To eliminate the first term on the right-hand side, we introduce a partition of unity $\{\psi_\alpha\}_{\alpha\in A}$ on $S$ such that $\supp \psi_\alpha = \overline{B_\alpha(\xi_\alpha,r_\alpha)}$. Next we change the variable in the formula for the solution
\begin{equation*}
   u(x,t) = \int_0^t\!\!\!\int_S n_i(\xi)\cdot \frac{\partial\Gamma(x-\xi,t-\tau)}{\partial \xi_i}\mu(\xi,\tau)\, \ud S_\xi \ud \tau,
\end{equation*}
using the mapping $\Phi_\alpha\colon \supp \psi_\alpha\to\mathbb{R}^n$, $\xi\xrightarrow{\Phi_\alpha} \lbrack\xi', f^{\alpha}(\xi')\rbrack$. With this in mind, we see that 
\begin{multline}\label{eq76}
	u(x,t) = \frac{1}{2}\frac{1}{(4\pi)^{\frac{n}{2}}} \int_0^t\!\!\!\int_{E^\alpha} \frac{\lbrack\nabla  f^{\alpha}(\xi'),-1\rbrack}{\sqrt{1 + \abs{\nabla f^{\alpha}}^2}} \cdot \frac{\lbrack x' - \xi',x_n - f^{\alpha}(\xi')\rbrack}{(t - \tau)^{\frac{n + 2}{2}}} e^{-\frac{\abs{x' - \xi'}^2 + \abs{x_n - f^{\alpha}(\xi')}^2}{4(t - \tau)}} \\
	\mu_{\alpha}(\xi',f^{\alpha}(\xi'),\tau)\sqrt{1 + \abs{\nabla f^{\alpha}}^2}\, \ud \xi'\, \ud \tau \\
	= c_n \int_0^t\!\!\!\int_{E^\alpha} \frac{\nabla  f^{\alpha}(\xi') \cdot \lbrack x' - \xi'\rbrack - x_n + f^{\alpha}(\xi')}{(t - \tau)^{\frac{n + 2}{2}}} e^{-\frac{\abs{x' - \xi'}^2 + \abs{x_n - f^{\alpha}(\xi')}^2}{4(t - \tau)}} \mu_{\alpha}(\xi',f^{\alpha}(\xi'),\tau)\, \ud \xi'\, \ud \tau,
\end{multline}
where $E^{\alpha} = \mathbb{R}^{n - 1} \cap \supp \psi(\xi')$ and $\cdot$ denotes the standard inner product. From the Mean Value Theorem it follows that
\begin{equation*}
	x_n - f^{\alpha}(\xi') = \nabla f^{\alpha}(\xi^*) \cdot \lbrack x' - \xi '\rbrack + x_n - f^{\alpha}(x').
\end{equation*}
Since
\begin{equation*}
	\abs{\nabla f^{\alpha}(\xi^*) \cdot \lbrack x' - \xi '\rbrack + x_n - f^{\alpha}(x')}^2 \geq \abs{\nabla f^{\alpha}(\xi^*) \cdot \lbrack x' - \xi '\rbrack}^2 +\abs{x_n - f^{\alpha}(x')}^2
\end{equation*}
we may apply Lemma \ref{l2} in \eqref{eq76}. It leads to
\begin{equation*}
	u(x,t) \leq c_n \int_0^t\!\!\!\int_{E^{\alpha}} \frac{1}{(t - \tau)^{\frac{n + 1}{2}}} e^{-\frac{c\abs{x' - \xi'}^2 + \abs{x_n - f^{\alpha}(x')}^2}{4(t - \tau)}} \mu_{\alpha}(\xi',f^{\alpha}(\xi'),\tau) \ud \xi'\, \ud \tau.
\end{equation*}
Next we repeat the proof of Lemma \ref{l7}. This concludes the proof.
\end{proof}

\end{document}